\documentclass[11pt]{article}
\usepackage{amsfonts}
\usepackage{amsmath}
\usepackage{amssymb}
\usepackage{amsthm}
\usepackage[usenames,dvipsnames]{color}
\usepackage{natbib}{\tiny }
\usepackage{tikz}
\usepackage[utf8]{inputenc}
\usepackage{mathtools}
\usepackage{graphicx}
\usepackage{xfrac}
\usepackage{thmtools, thm-restate}
\usepackage{hyperref}
\usepackage{cleveref}
\usepackage{cancel}
\usepackage{mathtools}
\usepackage{bbm}
\usepackage{bm}
\usepackage[sf,bf,tiny]{titlesec}
\titlelabel{\thetitle. \enspace}

\theoremstyle{definition}
\newtheorem{thm}{Theorem}

\newtheorem{cor}{Corollary}

\newtheorem{rem}{Remark}

\title{\Large \bf The star-shaped $\Lambda$-coalescent and Fleming-Viot process}

\author{\sc Robert Griffiths\thanks{Department of Statistics, University of Oxford, 1 South Parks Rd, Oxford, OX1 3TG, UK; email  griff@stats.ox.ac.uk; Corresponding author}
\and
\sc Shuhei Mano\thanks{The Institute of Statistical Mathematics, 10-3 Midori-cho, Tachikawa, Tokyo, 190-8562, Japan; email smano@ism.ac.jp}
}

\begin{document}
\maketitle

\section*{Abstract} 
The star-shaped $\Lambda$-coalescent and corresponding $\Lambda$-Fleming-Viot process where the $\Lambda$ measure has a single atom at unity are studied in this paper. The transition functions and stationary distribution of the $\Lambda$-Fleming-Viot process are derived in a two-type model with mutation. 
The distribution of the number of non-mutant lines back in time in the star-shaped  $\Lambda$-coalescent is found.
Extensions are made to a model with $d$ types, either with parent independent mutation or general Markov mutation, and an infinitely-many-types model when $d\to \infty$. An eigenfunction expansion for the transition functions is found which has polynomial right eigenfunctions and left eigenfunctions described by hyperfunctions. A further star-shaped model with general frequency dependent change is considered and the stationary distribution in the Fleming-Viot process derived. This model includes a star-shaped $\Lambda$-Fleming-Viot process with mutation and selection.
In a general $\Lambda$-coalescent explicit formulae for the transition functions and stationary distribution when there is mutation are unknown, however in this paper explicit formulae are derived in the star-shaped coalescent.
\bigskip

\noindent
{\bf MSC} 92D02

\noindent
{\bf Keywords:}
$\Lambda$-coalescent; $\Lambda$-Fleming-Viot process; Star-shaped coalescent.
\noindent
{\bf Running Head:} Star-shaped coalescent

\section{Introduction}
A Fleming-Viot process $\{\xi_t\}_{t\geq 0}$ representing the frequency of type 1 individuals in a population of two types has a generator ${\cal L}^\circ$ acting on functions in $C^2([0,1])$ described by 
\begin{equation}
{\cal L}^\circ g(x) = \int_0^1x\big (g(x(1-y)+y)-g(x)\big ) + (1-x)\big (g(x(1-y)) - g(x)\big )\frac{\Lambda(dy)}{y^2},
\label{genx:0}
\end{equation}
where $\Lambda$ is a finite measure on $[0,1]$. The population is partitioned at events of change by choosing a type 1 individual with probability $x$, or a type 2 individual with probability $1-x$; adding in an additional frequency of $y$ of the type chosen, then rescaling the frequencies to add to 1. The rate of reproduction events with offspring frequency $y$ is 
$y^{-2}\Lambda(dy)$. If $\Lambda$ has a single atom at $0$, then $\{\xi(t)\}_{t\geq 0}$ is a Wright-Fisher diffusion process.  \citet{BB2009} describe the $\Lambda$-Fleming-Viot process and discrete models whose limit gives rise to the process. \citet{EW2006} introduced a model where $\Lambda$ has a single point of increase in $(0,1]$ with a possible atom at zero as well. A model where there is mutation from type 2 individuals to type 1 individuals at rate $\theta_1/2$ and from type 1 to type 2 individuals at rate $\theta_2/2$ (with $\theta = \theta_1+\theta_2$) has an additional term added to the generator (\ref{genx:0});
\begin{equation}
{\cal L}g(x) = {\cal L}^\circ g(x) 
+ \cfrac{1}{2}(\theta_1 - \theta x)g^\prime (x).
\label{gen:0}
\end{equation}
A stationary distribution exists in this process because of recurrent mutation between types.

The $\Lambda$-coalescent is a random tree back in time which has multiple merger rates for a specific set of $2 \leq k\leq n$ edges merging while $n$ edges in the tree of
\begin{equation}
\lambda_{nk} = \int_0^1x^k(1-x)^{n-k}\frac{\Lambda (dx)}{x^2},\>k \geq 2.
\label{lnk:rates}
\end{equation}
After coalescence there are $n-k+1$ edges in the tree. 
The $\Lambda$-coalescent process was introduced by \citet{DK1999,P1999,S1999} and has been extensively studied \citep{P2002,B2009, GIM2014, TL2014}. The coalescent process is a moment dual to the $\Lambda$-Fleming-Viot process. That is, the $\Lambda$-coalescent tree describes the genealogy of a collection of individuals back in time in the $\Lambda$-Fleming-Viot process. See for example \citet{E2012}. In a model with mutation and generator (\ref{gen:0}) mutations occur at random on the edges of the tree. The $\Lambda$-coalescent is said to \emph{come down from infinity} if there is an enterance boundry at infinity. Then there is a finite number of edges at any time back in an initial tree beginning with an infinite number of leaves. \citet{BLG2006} showed that the $\Lambda$-coalescent comes down from infinity under the same condition that the continuous state branching process becomes extinct in finite time, that is when
\begin{equation}
\int_1^\infty\frac{dq}{\psi(q)} < \infty,
\label{comedown:0}
\end{equation}
where the Laplace exponent
 \begin{equation}
\psi(q) = \int_0^1 \big (e^{-qy} -1 +qy\big )y^{-2}\Lambda(dy).
\label{Laplaceexponent}
\end{equation}
\cite{G2014} shows that (\ref{genx:0}) can be written as a Wright-Fisher generator equation 
\begin{equation}
{\cal L}^\circ g(x) = \frac{1}{2}x(1-x)\mathbb{E}\Big [g^{\prime\prime}\big (x(1-W)+VW\big )\Big ]
\label{intro:2}
\end{equation}
where $W=UY$, $Y$ has distribution $\Lambda$ (assuming without loss of generality that $\Lambda$ is a probability measure), $U$ has a density $2u$, $u \in (0,1)$, $V$ is uniform on $(0,1)$, and $U,V,Y$ are independent. If $W=0$ the usual Wright-Fisher generator is obtained.

In this paper we consider a star-shaped $\Lambda$-coalescent tree and Fleming-Viot process when the measure has a single atom at 1, with 
$\Lambda(\{1\}) = 1$. This seems a simple model, however there are a number of aspects in the analysis that are not at all simple. In a general $\Lambda$-coalescent explicit formulae for the transition functions and stationary distribution when there is mutation are unknown, however in this paper explicit formulae are derived in the star-shaped coalescent.

This $\Lambda$-coalescent does not come down from infinity because a tree beginning with an infinite number of leaves still contains edges joining the leaves until an exponentially distributed time of rate 1 when all edges coalesce.
The generator of the star-shaped Fleming-Viot process with mutation included is
\begin{equation}
{\cal L}g(x) = x\big (g(1)-g(x)\big ) + (1-x)\big (g(0) - g(x)\big ) + \cfrac{1}{2}(\theta_1 - \theta x)g^\prime (x).
\label{gen:00x}
\end{equation}
In this model the entire population is replaced in a reproduction event.

A biological motivation for studying the star-shaped coalescent is that genealogies in exponentally growing populations tend to be star-shaped. A discussion of this property is in Section 4.2 of \cite{D2008}.  \citet{GT1998} show that in an exponentially growing population with large rate $\beta$, in  Kingman coalescent, beginning from $n$ individuals, is star-shaped with mean coalescence time asymptotically
\[
\beta^{-1}\Big (\log \beta + \log {n\choose 2}^{-1} - \gamma \Big ),
\]
where $\gamma=0.577216\ldots$ is Euler's constant. Coalescence time is short with a large growth rate, so mutation rates have to be of order $\beta/\log \beta$ for mutation to occur.

In Section 2 the transition functions and stationary distribution of the star-shaped Fleming-Viot process are found. An interesting eigenfunction expansion is found for the transition functions which involves polynomial right eigenvectors and left eigenvectors which are described by hyperfunctions. The appearance of hyperfunctions is likely to be related to the fact of not coming down from infinity and suggests the same complexity in other $\Lambda$-coalescent processes. An expansion of the transition functions at time $t$ partitioned by the number of population replacements in time $t$ is derived. In Section 3 the distribution of the number of non-mutant lines back in time in the star-shaped  $\Lambda$-coalescent is derived. The results in this paper are, to the knowledge of the authors, new and not a simple application of existing results in \citet{G2014} and other authors.
The model is extended to $d$ types in Section 4, either with parent independent mutation or general Markov mutation, and to infinitely many types as $d \to \infty$. 
 \citet{M2006} found the sampling distribution in the infinitely-many-types model using recursive equations. We also obtain this distribution from a different arguement. In Section 4 a model with very general frequency dependent change is considered. This includes a star-shaped model with mutation and selection where the stationary distribution is obtained.
\section{The star-shaped coalescent with two types}\label{Section 2}
In this section the transition distribution and stationary distribution of the two-type star-shaped process with mutation is obtained.
\begin{thm}\label{theorem:1}
Let $\{\xi(t)\}_{t\geq 0}$ be the frequency of type 1 individuals at time $t$, when $\xi(0)=x$, in the Fleming-Viot process with generator (\ref{gen:00x}).
Denote $p_{ij}(t)$, $i,j = 1,2$ as the probability of a change of type in a single line of the star-shaped coalescent in time $(0,t)$, and $q_i(t;x) = xp_{1i}(t) + (1-x)p_{2i}(t)$ as the probability a single individual chosen at time $t$  is of type $i = 1,2$, assuming in both cases that there is no replacement event in $(0,t)$.
Specifically, with $p = \theta_1/\theta$,
\[
p_{11}(t) = e^{-\theta t/2} + (1-e^{-\theta t/2})p,\>\>p_{21}(t) = (1-e^{-\theta t/2})p,
\]
and 
\[
q_1(t;x) = xe^{-\theta t/2} + (1-e^{-\theta t/2})p,\>\>
q_2(t;x) = 1 - q_1(t;x).
\]
 Then
\begin{equation}
\xi(t) = \begin{cases}
 q_1(t;x)& \text{with~probability~}e^{-t}\\
 p_{i1}(\tau)&i=1,2,\text{~with~probability~} \big (1 - e^{-t}\big )q_i(t-\tau;x),
 \end{cases}
\label{density:0}
\end{equation}
where in (\ref{density:0}), $\tau$ has a truncated exponential distribution with density
\begin{equation}
\big (1-e^{-t}\big )^{-1}e^{-\tau}, 0 < \tau < t.
\label{trunc}
\end{equation}
That is, $\xi(t)$ only depends on the last replacement event before $t$ at $\tau$, if a replacement event occurs.
The continuous part of the distribution of $\xi(t)$ has a density $f(\xi;x,t)$ equal to
\begin{equation}
\begin{cases}
\Big (p + e^{-\theta t/2}(x-p)\frac{1-p}{\xi-p}\Big )
\frac{2}{\theta}
\Big (\frac{\xi - p}{1 - p}\Big )^{2/\theta -1}
\frac{1}{1-p},
&\xi > p + (1-p)e^{-\theta t/2}\\
\Big (1 - p - e^{-\theta t/2}(x-p)\Big (1 - \frac{\xi}{p}\Big )^{-1}\Big )\frac{2}{\theta}\Big (1 - \frac{\xi}{p}\Big )^{2/\theta -1}\frac{1}{p},&\xi < p(1-e^{-\theta t/2})
 \end{cases}
\label{density:1}
\end{equation}
and zero in the region $p(1-e^{-\theta t/2}) < \xi <  p + (1-p)e^{-\theta t/2}$.
The stationary density of $\{\xi(t)\}_{t\geq 0}$ is
\begin{equation}
f(\xi) = \lim_{t \to \infty} f(\xi;x,t) = \begin{cases}
p\frac{2}{\theta}
\Big (\frac{\xi - p}{1 - p}\Big )^{2/\theta -1}
\frac{1}{1-p},
&\xi > p\\
(1-p)\frac{2}{\theta}\Big (1 - \frac{\xi}{p}\Big )^{2/\theta -1}\frac{1}{p},&\xi < p.
 \end{cases}
\label{density:2}
\end{equation}
\end{thm}
\begin{proof}
The formulae for $p_{11}(t)$, $p_{21}(t)$, $q_1(t)$ and $q_2(t)$ are elementary using an argument that if at least one mutation occurs along  single line in time $t$ with probability $1-e^{-\theta t/2}$ then the line is type 1 at time $t$ if the last mutation is type 1 with probability $p$.
 We use the approach of looking at the Lambda coalescent as a dual process to the Fleming-Viot process describing the population. Consider a star-shaped Lambda coalescent tree beginning with infinitely many leaves at time $t$. Complete coalescence occurs with in the tree at rate 1, giving an exponential time $T$ with rate 1 to the time to the most recent common ancestor. 
 Let $Z(t)$ be the number of type 1 individuals in a sample of $n$ individuals taken at time $t$.
There are two possibilities for the time of coalescence of the $n$ individuals: $T < t$ or $T \geq t$. The probability that $T < t$ and there are $z$ individuals of type 1 in the sample is
\begin{equation}
\sum_{i=1}^2\int_0^tq_i(t-\tau;x){n\choose z}p_{i1}(\tau)^z(1-p_{i1}(\tau))^{n-z}e^{-\tau}d\tau
\label{case:a}
\end{equation}
and the probability that $T \geq t$ and there are $z$ individuals of type 1 in the sample is
\begin{equation}
e^{-t}{n\choose z}q_1(t;x)^z\big (1- q_1(t;x)\big )^{n-z}.
\label{case:b}
\end{equation}
Let $\xi(t) = \lim_{n\to \infty}\frac{Z(t)}{n}$. The distribution of $\xi(t)$, (\ref{density:0}), can be read off from (\ref{case:b}) and (\ref{case:a}) by using the law of large numbers in the binomial distribution. 
The first case in (\ref{density:0}) is an atom in the distribution, and the second case leads to a density function.
 In the second case there are two disjoint regions for the density.
\begin{equation}
\begin{cases}
\text{If~}p_{11}(\tau) = \xi&\text{then~}e^{-\theta\tau/2} = \frac{\xi-p}{1-p},\>\xi > p\\
\text{If~}p_{21}(\tau) = \xi&\text{then~}e^{-\theta\tau/2} = 1 - \frac{\xi}{p},\>\xi < p.
\end{cases}
\label{basiceq:0}
\end{equation}
In addition 
\[
e^{-\theta\tau/2} > e^{-\theta t/2}
\]
must be satisfied in both the regions. That is, in the two regions,
\[
\begin{cases}
\frac{\xi-p}{1-p} > e^{-\theta t/2}&\text{if~}\xi > p\\
 1 - \frac{\xi}{p} > e^{-\theta t/2}&\text{if~}\xi < p.
\end{cases}
\]
Evaluating the density with a change of variable $\tau \to \xi$ from (\ref{basiceq:0}) gives (\ref{density:1}).
The limiting density as $t\to \infty$ in (\ref{density:1}) is clearly (\ref{density:2}). 
A direct verification that $\mathbb{E}\big [{\cal L}g(\xi)\big ] = 0$, $g\in C^2((0,1))\cap C([0,1])$ (the space of continuous functions in $[0,1]$ which have continuous second derivatives), when expectation is in the stationary distribution can also be shown. Applying the generator (\ref{gen:00x}) to $g(x)$ and taking expectations, 
\begin{equation}
\mathbb{E}\big [{\cal L}g(\xi)\big ] = pg(1) + (1-p)g(0) - \mathbb{E}\big [g(\xi)\big ] + \frac{\theta}{2}\mathbb{E}\big [(p-\xi)g^\prime (\xi)\big ].
\label{verify:0}
\end{equation}
Calculating the last term using integration by parts,
\begin{eqnarray}
&&\frac{\theta}{2}\mathbb{E}\big [(p-\xi)g^\prime (\xi)\big ]
\nonumber \\
 &&~~=
 -\frac{p}{(1-p)^{2/\theta}}\int_p^1(\xi-p)^{2/\theta}g^\prime(\xi)d\xi 
 + \frac{1-p}{p^{2/\theta}}\int_0^p(p-\xi)^{2/\theta}g^\prime(\xi)d\xi\nonumber \\
 &&~~=
  -\frac{p}{(1-p)^{2/\theta}}(\xi-p)^{2/\theta}g(\xi)\Big |_p^1
 + \frac{1-p}{p^{2/\theta}}(p-\xi)^{2/\theta}g(\xi)\Big |_0^p
 \nonumber \\
 &&~~
   +\frac{p}{(1-p)^{2/\theta}}\frac{2}{\theta}\int_p^1(\xi-p)^{2/\theta-1}g(\xi)d\xi
 + \frac{1-p}{p^{2/\theta}}\frac{2}{\theta}\int_0^p(p-\xi)^{2/\theta-1}g(\xi)d\xi
 \nonumber \\
 &&=
 -pg(1)-(1-p)g(0) + \mathbb{E}\big [g(\xi)\big ].
 \label{verify:1}
\end{eqnarray}
Substituting (\ref{verify:1}) into (\ref{verify:0}) shows that
 $\mathbb{E}\big [{\cal L}g(\xi)\big ] = 0$.
\end{proof}
Throughout the paper $\mathbb{E}\big [\cdot\big ]$ usually denotes expectation in the stationary distribution, and 
 $\mathbb{E}_x\big [\cdot\big ]$ denotes $\mathbb{E}\big [\cdot\mid \xi(0)=x\big ]$. 

\begin{cor}
Moments of $\xi(t)$ are, for $n=0,1,\ldots$
\begin{eqnarray}
\mathbb{E}_x\big [\big (\xi(t)-p\big )^n\big ]
&=&e^{-t(1+n\theta /2)}(x-p)^n
\nonumber \\
&&+
\frac{2/\theta}{n+2/\theta}\Big [ p(1-p)^n+(-1)^{n}(1-p)p^n\Big ]
\nonumber \\
&&~~~~~~~~~~~\times\Big [1-e^{-t\big (1+n\theta/2\big )}\Big ]
\nonumber \\
&&+
\frac{(x-p)e^{-\theta t/2}(2/\theta)}{n-1+2/\theta}\Big [ (1-p)^{n}-(-1)^{n}p^{n}\Big ]
\nonumber \\
&&~~~~~~~~~~~
\times\Big [1-e^{-t\big (1+(n-1)\theta/2\big )}\Big ]
\label{moments:0}
\end{eqnarray}
\end{cor}
\begin{proof}
 Preliminary calculations that are needed with $a=2/\theta$, and $a=2/\theta - 1$ are:
\begin{eqnarray}
&&\int_{p + (1-p)e^{-\theta t/2}}^1(\xi - p)^n\Big (\frac{\xi - p}{1 - p}\Big )^{a-1}\frac{1}{1-p}d\xi
= \frac{(1-p)^n}{n+a}\Big [1 - e^{-(n+a)\theta t/2}\Big ]\nonumber \\
&&\int_0^{p(1-e^{-\theta t/2})}(\xi - p)^n\Big (1 - \frac{\xi}{p}\Big )^{a -1}\frac{1}{p}d\xi
= \frac{(-p)^n}{n+a}\Big [1 - e^{-(n+a)\theta t/2}\Big ].\label{integrals:0}
\end{eqnarray}
Considering (\ref{integrals:0}) with the different parts of the distribution of $\xi(t)$ gives (\ref{moments:0}).
\end{proof}
\begin{rem}\label{Remark 1}
 A representation for the stationary distribution (\ref{density:2}) is the following.
 Let $\eta$ be a random variable with density $(2/\theta)\eta^{2/\theta -1}$, $0 < \eta < 1$. Then with probability $p$, $\xi = p(1-\eta) + \eta$ and with probability $1-p$, $\xi = p(1-\eta)$.
\end{rem}
\subsection{Eigenvalues and Eigenvectors of the generator}
\cite{G2014}, Corollary 4 shows that the right eigenvectors of the general $\Lambda$ coalescent generator are polynomials and the eigenvalues are
$\lambda_n = \frac{1}{2}\mathbb{E}\big [n\big ((n-1)(1-W)^{n-2} + \theta\big )\big ]$, where $W$ is described in the Introduction. 
In our generator (\ref{gen:0}), the eigenvalues are $\lambda_1 = \theta/2$ and for $n=2,\ldots$
\begin{eqnarray}
\lambda_n &=& \frac{1}{2}\mathbb{E}\big [n\big ((n-1)(1-W)^{n-2} + \theta\big )\big ]
\nonumber \\
&=& 
\frac{1}{2}\int_0^12un(n-1)(1-u)^{n-2}du + \frac{1}{2}n\theta
\nonumber \\
&=& 1 + n\theta/2.
\label{eigenW:0}
\end{eqnarray}
Equation (\ref{moments:0}) is consistent with right polynomial eigenvectors and eigenvalues (\ref{eigenW:0}) because
\[
\mathbb{E}_x\big [\xi(t)\big ] = p + (x-p)e^{-\theta t/2}
\]
and
\begin{equation*}
\mathbb{E}_x\big [\xi(t)^n] = e^{-t(1+n\theta/2)}x^n + R_{n-1}(x),
\end{equation*}
where $R_{n-1}(x)$ is a polynomial of degree $n-1$ in $x$.
\begin{thm}
The eigenvalues of ${\cal L}$, defined in (\ref{gen:0}) are $\lambda_1 = n\theta/2$, $\lambda_n = 1 + n\theta/2$, $n\geq 2$. The right eigenvectors are polynomials $\{P_n(x)\}_{n=1}^\infty$, taken to be monic, which satisfy
\[
{\cal L}P_n(x) = -\lambda_nP_n(x).
\]
The explict form for these polynomials is 
 $P_1(x) = x-p$, and for $n \geq 2$
\begin{equation}
P_n(x) = (x-p)^n + c_{n1}(x-p) + c_{n0},
\label{eigen:0}
\end{equation}
where
\begin{eqnarray}
c_{n0} &=& -\frac{2/\theta}{n+2/\theta}\Big [ p(1-p)^n+(1-p)(-p)^n\Big ]
\\
c_{n1} &=& \frac{2/\theta}{n-1+2/\theta}\Big [(-p)^{n}- (1-p)^{n}\Big ].
\label{eigencoefficients}
\end{eqnarray}
\end{thm}
\begin{proof}
A general Gram-Schmidt construction is the following. Suppose for a generator ${\cal L}$
that ${\cal L}x^n$ is a polynomial of degree $n$ in $x$, for $n\in \mathbb{Z}_+$.
${\cal L}1 = 0$, $\lambda_0 = 0$, $P_0(x)=1$;  and recursively we can write
\[
{\cal L}x^n = -\lambda_n x^n  + \sum_{k=0}^{n-1}a_{nk}P_k(x),
\]
for a set of constants $\{a_{nk}\}$. Now let
\[
P_n(x) = x^n + \sum_{k=0}^{n-1}c_{nk}P_k(x),
\]
where the constants $\{c_{nk}\}$ have to be chosen so that
\[
{\cal L}P_n(x) = -\lambda_n P_n(x).
\]
Now
\begin{eqnarray*}
{\cal L}P_n(x) &=& -\lambda_nx^n + \sum_{k=0}^{n-1}a_{nk}P_k(x)
-\sum_{k=0}^{n-1}\lambda_k c_{nk}P_k(x)
\nonumber \\
&=& -\lambda_n (x^n + \sum_{k=0}^{n-1}c_{nk}P_k(x)).
\end{eqnarray*}
Therefore we must choose $-\lambda_nc_{nk} = a_{nk} - \lambda_kc_{nk}$ or
\[
c_{nk}=\frac{a_{nk}}{\lambda_k - \lambda_n},
\]
allowing a recursive construction.
In our case we are considering polynomials in $x-p$, however the idea is the same. Clearly ${\cal L}(x-p)=-\theta/2(x-p)$, and for $n\geq 2$
\begin{eqnarray*}
{\cal L}(x-p)^n &=& x(1-p)^n + (1-x)(-p)^n - (x-p)^n - (x-p)^nn\theta/2
\nonumber \\
&=& -(1+n\theta/2)(x-p)^n + (x-p)\big ((1-p)^n-(-p)^n\big )
\nonumber \\
&&~~~~~~~~+
p(1-p)^n + (1-p)(-p)^n.
\end{eqnarray*}
In the context of the recursive construction $\lambda_n = 1 + n\theta/2$,
$a_{n1} = (1-p)^n - (-p)^n$, $a_{n0} = p(1-p)^n + (1-p)(-p)^n$.
Therefore
\begin{eqnarray*}
c_{n1} &=& \frac{(1-p)^n - (-p)^n}{\lambda_1-\lambda_n}
\nonumber \\
&=&
\frac{ (-p)^n - (1-p)^n}{1+(n-1)\theta /2},
\nonumber \\
c_{n0} &=&-\frac{(1-p)(-p)^n + p(1-p)^n}{1+n\theta /2}
\end{eqnarray*}
and
\[
P_n(x) = (x-p)^n + c_{n1}(x-p) + c_{n0}.
\]
This agrees with (\ref{eigen:0}).
\end{proof}
The left eigenvectors $\{Q_n(x)\}_{n=1}^\infty$ are functions informally satisfying
\begin{equation}
f(\xi;x,t) = f(\xi)\Big \{1 + \sum_{n=1}^\infty e^{-\lambda_n t}P_n(x)Q_n(\xi)\Big \}
\label{c:3a0}
\end{equation}
with the atom included in the distribution, where $f(\xi)$ is the stationary density (\ref{density:2}). In fact the continuous part of the distribution has a density
\[
f(\xi;x,t) = f(\xi)\Big \{1 + e^{-\lambda_1 t}P_1(x)Q_1(\xi)\Big \}
\]
in the region $\{0 < \xi < p(1-e^{-\theta t/e})\}\cup
\{p + (1-p)e^{-\theta t/2}< \xi < 1\}$, however a difficulty is the restricted support of the distribution.
The formal way to specify an expansion which includes the atom is to require
\begin{equation}
\mathbb{E}_x\big [g\big (\xi(t) \big )]
= \sum_{n=0}^\infty
e^{-\lambda_n t}\mathbb{E}\big [g(\xi)Q_n(\xi)\big ]P_n(x),
\label{c:3a}
\end{equation}
for suitable functions $g\in C^\infty((0,1))\cap C([0,1])$. The eigenfunctions have a biorthogonality property that
\begin{equation}
\mathbb{E}\big [P_m(\xi)Q_n(\xi)\big ] = \delta_{mn}.
\label{bio:0}
\end{equation}
$Q_1(\xi)$ is defined in (\ref{split:q1}).
Although the transition density is relatively simple it is difficult to make either (\ref{c:3a0}) or (\ref{c:3a}) precise. The difficulty is because of the \emph{dust} like property of the process and the way the continuous part of the transition density depends on $t$.
There is an interesting solution for the left eigenvectors in terms of hyperfunctions, which are generalized functions \citep{K1988}. Denote $\widetilde{Q}_n(\xi) = Q_n(\xi)f(\xi)$; then
$
\mathbb{E}\big [g(\xi)Q_n(\xi)\big ] = \langle g,\widetilde{Q}_n\rangle
$, $n\geq 2$
is regarded as an integral defined by the action of a hyperfunction $\widetilde{Q}_n(\xi)$.
Another difficulty is that if $\theta/2 < 1$ then $\mathbb{E}\big [Q_1(\xi)\big ]$ is not well defined because of a singularity at $p$, whereas it is required that $\mathbb{E}\big [g(\xi)Q_1(\xi)\big ]$ is well defined when $g$ has a power series expansion about $p$. Recalling the representation of $\xi$ in terms of $\eta$ in Remark 1 we define
\[
\mathbb{E}\big[Q_1(\xi)\big ] = \lim_{\epsilon \to 0}
\mathbb{E}\big[Q_1(\xi)I\{\eta > \epsilon\}\big ] = 0
\]
in the sense of a Cauchy principal value. $\mathbb{E}\big[g(\xi)Q_1(\xi)\big ]$ is similarly interpreted as a Cauchy principal value. This allows 
(\ref{hyperexpansion:0}) to hold for constant functions $g$.

A hyperfunction is specified by a pair $(h,k)$, where $h$ and $k$ are (complex) holomorphic functions defined respectively on the upper and lower complex half-planes. The hyperfunction is represented by its actions in complex integrals on test functions with weight $h-k$. A technique to deal with an essential singularity at zero of $h$ is to take a hyperfunction $(h,h)$. The hyperfunction we use behaves in integrals as $\delta(\xi-p)$.
 We need to use derivatives of this function in integrals which cannot be taken in usual integration theory.  A classical first example of a hyperfunction is how to express a Dirac delta function as a hyperfunction in integration. As a brief example consider the Dirac delta function $\delta (\xi)$.  (For simplicity the function is centered at zero.) 
Cauchy's formula is
\begin{equation}
\frac{1}{2\pi i}\oint \frac{\varphi(\xi)}{\xi}d\xi = \varphi(0), \qquad
i\equiv\sqrt{-1}.
\label{cauchy}
\end{equation}
Deform the contour of the integral as 
\begin{eqnarray*}
&&-\frac{1}{2\pi i}\int_{a+i0}^{b+i0}\frac{\varphi(\xi)}{\xi}d\xi 
+\frac{1}{2\pi i}\int_{a-i0}^{b-i0}\frac{\varphi(\xi)}{\xi}d\xi\\
&&~~= \int_a^b \left(-\frac{1}{2\pi i}\right)
\left(\frac{1}{\xi+i0}-\frac{1}{\xi-i0}\right)\varphi(\xi)d\xi,
\end{eqnarray*}
where $a$ and $b$ are holomorphic points of $1/\xi$.
The Dirac delta function as a hyperfunction is $(1/(2\pi i\xi),1/(2\pi i\xi))$.
Equation (\ref{cauchy}) is concisely written as
$\langle \delta,\varphi \rangle=\varphi(0)$,
where
\begin{equation*}
\delta(\xi)=
\left(\frac{1}{2\pi i}\frac{1}{\xi},\frac{1}{2\pi i}\frac{1}{\xi}\right).
\end{equation*}
Differentiation  can be justified and the explicit expression for the $j$th derivative is
\begin{equation*}
\delta^{(j)}(\xi)=
\left(\frac{1}{2\pi i}\frac{(-1)^j j!}{\xi^{j+1}},\frac{(-1)^j j!}{2\pi i}
\frac{1}{\xi^{j+1}}\right),
\qquad j=1,2,...
\end{equation*}
Integration by parts gives
$\langle \delta^{(j)},\varphi \rangle
=(-1)^j \langle \delta,\varphi^{(j)} \rangle$,
which is the Cauchy-Goursat formula
\begin{equation*}
\frac{1}{2\pi i}\oint \frac{\varphi(\xi)}{\xi^{j+1}}d\xi 
= \frac{\varphi^{(j)}(0)}{j!}.
\end{equation*}
Hyperfunctions are a generalization of Schwartz distributions, which are linear operators defined by their actions on test functions. Our eigenfunctions could be defined by Schwartz distributions,  however we prefer the elegance of hyperfunction theory.
\begin{thm}
Let $g \in C^\infty((0,1))\cap C([0,1])$, then
\begin{eqnarray}
&&\mathbb{E}_x\big [g(\xi(t))\big ] = \mathbb{E}\big [g(\xi)\big ]
+ e^{-\lambda_1t}\mathbb{E}\big [g(\xi)Q_1(\xi)\big ]P_1(x)
\nonumber \\
&&~~~~~~~~~~~~~~~~~~~+ \sum_{n=2}^\infty e^{-\lambda_nt}\langle g,\widetilde{Q}_n\rangle P_n(x),
\label{hyperexpansion:0}
\end{eqnarray}
where 
\begin{equation}
Q_1(\xi) = \begin{cases}
 \frac{1}{p}\Big (\frac{\xi -p}{1-p}\Big )^{-1},&\xi > p\\
 -\frac{1}{1-p}\Big (1-\frac{\xi}{p}\Big )^{-1},&\xi< p;
 \end{cases}
 \label{split:q1}
 \end{equation}
the left eigenfunctions  times the stationary density when $n\geq 2$ are hyperfunctions
\begin{equation}
\widetilde{Q}_n(\xi) = \frac{(-1)^n}{n!}\delta^{(n)}(\xi-p);
\label{hyperexpansion:1}
\end{equation}
and $\delta^{(n)}(\xi-p)$ is the $n$th derivative, in a hyperfunction sense, of the Dirac delta function. Explicitly
\begin{equation}
\langle g,\widetilde{Q}_n\rangle  = \frac{g^{(n)}(p)}{n!}.
\label{hyperexpansion:2}
\end{equation}
 There is biorthogonality of the right and left eigenvectors
\begin{equation}
\langle P_m,\widetilde{Q}_n\rangle = \delta_{mn}.
\label{hyperexpansion:3}
\end{equation}
\end{thm}
\begin{proof}
The left side of (\ref{hyperexpansion:0}) evaluates to
\begin{eqnarray}
\mathbb{E}_x\big [g\big (\xi(t) \big )]
&=&
e^{-t}g(q_1(t;x)) 
\nonumber \\
&+& \sum_{i=1}^2
\int_0^tg(p_{i1}(\tau))q_i(t-\tau;x)e^{-\tau}d\tau.
\label{c:3b}
\end{eqnarray}
If $g$ has a power series expansion in $(0,1)$ with a radius of convergence 1, then the first term on the right side of (\ref{c:3b}) is equal to 
\begin{eqnarray}
&&e^{-t}\sum_{n=0}^\infty \frac{g^{(n)}(p)}{n!}(x-p)^ne^{-n\theta t/2}\nonumber \\
&&=e^{-t}g(p) + e^{-t}g^{(1)}(p)e^{-\lambda_1 t}P_1(x)
\nonumber \\
&&+\sum_{n=2}^\infty \frac{g^{(n)}(p)}{n!}e^{-\lambda_n t}
\Bigg (P_n(x) - c_{n1}P_1(x) - c_{n0}\Bigg )  
\label{c:3c}
\end{eqnarray}
because 
\[
g(q_1(t;x)) = g(p + (x-p)e^{-\theta t/2}).
\]
The second two terms on the right of (\ref{c:3b}) contain only linear terms in $x-p$.  Now considering the right side of (\ref{hyperexpansion:0}), (\ref{c:3c}), and equating coefficients of $P_n(x)$, $n\geq 2$, it must be that 
\begin{equation}
\langle g,\widetilde{Q}_n \rangle = \frac{g^{(n)}(p)}{n!}.
\label{c:3d}
\end{equation}
A property of integration by parts in the hyperfunction integral is
\[
\langle g,\delta_p^{(n)}\rangle = (-1)^n \langle g^{(n)},\delta_p\rangle = (-1)^ng^{(n)}(p),
\]
so (\ref{c:3d}) is satisfied.
There still needs to be a check on what happens with the terms $P_1(x)$ and the constants in (\ref{c:3b}) and (\ref{c:3c}).
Consider the two terms in the sum on the right of (\ref{c:3b}). A calculation shows that the sum is
\begin{eqnarray}
&&\int_0^tg\big (p + (1-p)e^{-\theta \tau/2}\big )
\big (p + (x-p)e^{-\theta(t-\tau)/2}\big )e^{-\tau}d\tau
\nonumber \\
&&+\int_0^tg\big ((p-pe^{-\theta\tau/2}\big )
\big (1-p-(x-p)e^{-\theta(t-\tau)/2}\big )e^{-\tau}d\tau
\nonumber \\
&&~= (1-e^{-t})g(p) + (x-p)g^{(1)}(p)e^{-\lambda_1 t}(1-e^{-t})
\nonumber \\
&&~-\sum_{n=2}^\infty\frac{g^{(n)}(p)}{n!}c_{n0}\big (1-e^{-\lambda_n t}\big )
\nonumber \\
&&~- (x-p)\sum_{n=2}^\infty \frac{g^{(n)}(p)}{n!}c_{n1}
\big (e^{-\lambda_1t}-e^{-\lambda_nt}\big ).
\label{twosum:0}
\end{eqnarray}
Adding (\ref{c:3c}) and (\ref{twosum:0}) the right side of (\ref{c:3b})  is equal to
\begin{eqnarray}
&&g(p) - \sum_{n=2}^\infty\frac{g^{(n)}(p)}{n!}c_{n0}
+e^{-\lambda_1t}P_1(x)\Big \{g^{(1)}(p) - \sum_{n=2}^\infty 
\frac{g^{(n)}(p)}{n!}c_{n1}\Big \}~~~~~~~~
\label{rightside:a}\\
&&~~~~~~~~~~~~~~~~~~~~~~~~~~~~~+\sum_{n=2}^\infty e^{-\lambda_nt}\frac{g^{(n)}(p)}{n!}P_n(x).
\label{rightside:0}
\end{eqnarray}
To check the linear terms (\ref{rightside:a}) with the left side of (\ref{c:3b}) consider
\begin{eqnarray}
\mathbb{E}\big [g(\xi)\big ] &=&
\frac{2}{\theta}(1-p)p^{-2/\theta}\int_0^pg(\xi)(p-\xi)^{2/\theta - 1}d\xi
\nonumber \\
&&+ \frac{2}{\theta}p(1-p)^{-2/\theta}\int_p^1g(\xi)(\xi-p)^{2/\theta - 1}d\xi
\nonumber \\
&=& g(p) + \frac{2}{\theta}\sum_{n=1}^\infty 
\frac{g^{(n)}(p)}{n!}\cdot
\frac{(1-p)(-p)^n + p(1-p)^n}{1 + n\theta/2}
\nonumber \\
&=& g(p) - \sum_{n=2}^\infty 
\frac{g^{(n)}(p)}{n!}c_{n0}
\label{check:5a}
\end{eqnarray}
and
\begin{eqnarray}
\mathbb{E}\big [g(\xi)Q_1(\xi)\big ]
&=&
\mathbb{E}\big [\big (g(p) + g(\xi) - g(p)\big )Q_1(\xi)\big ]
\nonumber \\
&=& 0 + \mathbb{E}\big [\big (g(\xi) - g(p)\big )Q_1(\xi)\big ]
\nonumber \\
 &=&
-\frac{2}{\theta}p^{1-2/\theta}\int_0^p\big (g(\xi)-g(p)\big )(p-\xi)^{2/\theta - 2}d\xi
\nonumber \\
&&+ \frac{2}{\theta}(1-p)^{1-2/\theta}\int_p^1\big (g(\xi)-g(p)\big )(\xi-p)^{2/\theta - 2}d\xi
\nonumber \\
&=& \frac{2}{\theta}\sum_{n=1}^\infty 
\frac{g^{(n)}(p)}{n!}\cdot
\frac{(1-p)^n - (-p)^n}{1 + (n-1)\theta/2}
\nonumber \\
&=& g^{(1)}(p) - \sum_{n=2}^\infty 
\frac{g^{(n)}(p)}{n!}c_{n1}.
\label{check:5}
\end{eqnarray}
Therefore (\ref{rightside:a}) is equal to 
\[
\mathbb{E}\big [g(\xi)] + e^{-\lambda_1t}
\mathbb{E}\big [g(\xi)Q_1(\xi)\big ] P_1(x)
\]
showing that the left and right sides of (\ref{c:3b}) agree.
Checking the biorthogonality (\ref{hyperexpansion:3}), for $n \geq 2, m \geq 1$, 
\[
\langle P_m,\widetilde{Q}_n \rangle
 = \frac{(-1)^n}{n!}\langle P_m,\delta^{(n)}_p \rangle
  = \frac{1}{n!}\langle P^{(n)}_m,\delta_p \rangle
  = \delta_{mn}.
\]
For orthogonality of $P_m$, $m\geq 2$ with the first function $Q_1$ let 
$g = P_m$ in (\ref{check:5}), then
\[
\mathbb{E}\big [P_m(\xi)Q_1(\xi)\big ] = P_m^{(1)}(p) - c_{m1}\frac{1}{m!}P^{(m)}_m(p) = c_{m1}-c_{m1}=0.
\]
\end{proof}
\section{Non-mutant ancestral lines}\label{Section 3}
Let $A_n^\theta(t)$ be the number of non-mutant ancestral lines at time $t$ back in a sample of $n$ individuals taken at a current time. In the Kingman coalescent 
$\{A_n^\theta(t)\}_{t\geq 0}$ is a simple death process, but in the $\Lambda$-coalescent multiple deaths can occur because of multiple mergers. Another difference is if the $\Lambda$-coalescent does not come down from infinity, then  $A^\theta_\infty(t) = \infty$ for $t>0$.
The death rates in the star-shaped coalescent are for $i\geq 1$ and $j\ne i$
\begin{equation*}
r_{ij} = 
\begin{cases}
i\theta/2&j=i-1\\
1&j=1.
\end{cases}
\end{equation*}
A similar argument to that in Section \ref{Section 2} is used to obtain the distribution of $A_n^\theta(t)$. We omit details of the calculation.
\begin{thm}
 Denote the probability of no mutation along a line of length $t$ by $p(t) = e^{-\theta t/2}$, $t \geq 0$ and let $T$ be the time to coalescence in the population.  Initially $A_n^\theta(0)=n$.
 Then
\begin{eqnarray}
P(A_n^\theta(t)=j) &=& P(A_n^\theta(t)=j,T>t) 
\nonumber \\
&=& {n\choose j}p(t)^j(1-p(t))^{n-j}e^{-t},\>1 < j \leq n;
\label{An:j} \\
P(A_n^\theta(t)=1) &=& P(A_n^\theta(t)=1,T>t) +P(A_n^\theta(t)=1,T \leq t)
\nonumber \\
&=& np(t)(1-p(t))^{n-1}e^{-t}
+ \int_0^t \big (1 - (1-p(\tau))^n\big )p(t-\tau)e^{-\tau}d\tau
\nonumber \\
&=& np(t)(1-p(t))^{n-1}e^{-t}
+ \sum_{k=1}^n{n\choose k}(-1)^{k-1}\frac{
p(t)-e^{-t}p(t)^k
}{1 + (k-1)\theta/2};~~~~~~
\label{An:1} \\
P(A_n^\theta(t)=0) &=& 1 - \big ( 1 - (1-p(t))^n\big )e^{-t}
- \int_0^t\big (1 - (1-p(\tau))^n\big )p(t-\tau)e^{-\tau}d\tau
\nonumber \\
&=&
1 
- \sum_{k=1}^n{n\choose k}(-1)^{k-1}\frac{
p(t) - (k-1)(\theta/2)e^{-t}p(t)^k
}{1 + (k-1)\theta/2}.~~~~~~~
\label{An:0}
\end{eqnarray}
In the limit as $n \to \infty$ the proportion of non-mutant lines at time $t$ converges to $p(t)$ if $T>t$; if $T\leq t$ then there is a single non-mutant line if there is no mutation on the line between $T$ and $t$ or no non-mutant lines if there is mutation  between $T$ and $t$. Formally\begin{eqnarray*}
\lim_{n\to \infty}P(A^\theta_n(t)=1) &=& 
\frac{e^{-\theta t/2}-e^{-t}}{1-\theta/2},
\nonumber \\
\lim_{n\to \infty}P(A^\theta_n(t)=0) &=& 1 -\frac{e^{-\theta t/2}-e^{-t}}{1-\theta/2}
\end{eqnarray*}
and
\[
\lim_{n\to \infty}P(A^\theta_n(t)\geq 2) = e^{-t},\text{~with~}
\lim_{n\to \infty}P(A^\theta_n(t)=k) = 0,\>k<\infty.
\]
The loss of mass in the last equation is consistent with the star-shaped coalescent not coming down from infinity.
\end{thm}
Let $T^\theta$ be the time until there are no non-mutant lines. Then
$P(T^\theta > t) = P(A_n^\theta(t) > 0)$. The mean of $T^\theta$ is
\begin{eqnarray}
\mathbb{E}\big [T^\theta\big ] &=& \int_0^\infty P(T^\theta > t)dt
\nonumber \\
&=& \int_0^\infty \big (1 - (1-p(t))^n\big )e^{-t}dt
\nonumber \\
&&~~
+ \int_0^\infty \int_0^t\big (1 - (1-p(\tau))^n\big )p(t-\tau)e^{-\tau}d\tau dt
\nonumber \\
&=& (1 + 2/\theta)\big (1 - n!(1+2/\theta)^{-1}_{(n)}\big ),
\label{meanT:0}
\end{eqnarray}
after evaluation of the integrals.  Notation is that  $a_{(n)} := a(a+1)\cdots (a+n-1)$ for $a \in \mathbb{R}$ and $n$ a non-negative integer.

There is a duality relationship between $\xi(t)$ and $A_n^\theta(t)$. 
Consider the probability that all individuals in a sample of $n$ taken at time $t$ are of type  1.
Knowing that at least one mutation has occurred on a line, the probability that it is of type 1 at forward time $t$ is $p$. 
If $T > t$ there are $A_n^\theta(t)$ lines at time $t$ back where the origin is with a frequency of $x$, then the probability of all being of type 1 is $x^{A_n^\theta(t)}$. If $T < t$ and $A_n^\theta(t) = 1$ then the single line at the origin must be of type 1 with probability $x$ and the mutant lines at time $T$ must be of type 1 with probability $p^{n-A_n^\theta(T)}$. If $T < t$ and $A_n^\theta(t) = 0$, then the mutant lines at $T$ and the mutant single line from 0 to $T$ must be type 1 with probability $p\cdot p^{n-A_n^\theta(t)}$.
Therefore
\begin{eqnarray}
\mathbb{E}_x\big [\xi(t)^n\big ] &=&
\mathbb{E}\big [x^{A_n^\theta(t)}p^{n-A_n^\theta(t)},T > t\big ]
\nonumber \\
&&+\mathbb{E}\big [p^{n+1-A_n^\theta(T)},\{A_n^\theta(t)=0,T < t\}\big ]
\nonumber \\
&&+\mathbb{E}\big [xp^{n-A_n^\theta(T)},\{A_n^\theta(t)=1,T < t\}\big ].
\label{dual:0}
\end{eqnarray}
Letting $t \to \infty$ in (\ref{dual:0}) yields
\begin{equation}
\lim_{t\to \infty}\mathbb{E}_x\big [\xi(t)^n\big ] 
= p^{n+1}\mathbb{E}\big [(1/p)^{A_n^\theta(T)}\big ].
\label{genfun:2}
\end{equation}
The right side of (\ref{genfun:2}) provides a generating function for the number of ancestral lines at the time of coalescence, since we know the $n$th moment in the stationary distribution on the left side, from (\ref{check:5a}), to be
\[
\mathbb{E}\big [\xi^n\big ]
= p^n - \sum_{k=2}^n{n\choose k}p^{n-k}c_{k0}.
\]
\subsection{Spectral decomposition of $P(A_n^\theta(t)=j)$}
In the Kingman coalescent the spectral decomposition of the probability distribution of $A_n^\theta(t)$ is well known \citep{G1980,T1984}. There is also such a decomposition for the line distribution in the star-shaped coalescent.
\begin{thm}
For $n=1,\ldots; j=0,\ldots,n$,
\begin{equation}
P(A_n^\theta(t) = j) = \sum_{k=0}^ne^{-\lambda_kt}Q_n^{(k)}P_j^{(k)},
\label{spectral:0}
\end{equation}
where $\lambda_0=0$, $\lambda_1 = \theta/2$, $\lambda_k=1+k\theta/2$, $k\geq 2$;
$P_0^{(0)}= 1$, $P_0^{(1)}=-1$,
$Q_n^{(0)}=1$, $n \geq 1$; 
and for $n \geq 1$
\begin{eqnarray}
Q^{(1)}_n &=& \sum_{i=1}^n{n\choose i}\frac{(-1)^{i-1}}{1 + (i-1)\theta/2};
\nonumber \\
Q^{(k)}_n &=& {n\choose k}(-1)^{k-1}\frac{(k-1)(1+k\theta/2)}{1 + (k-1)\theta/2},\>
1 < k \leq n;
\nonumber \\
P^{(k)}_0 &=& -\frac{\theta/2}{1+k\theta/2},\> 2 \leq k \leq n;
\nonumber \\
P^{(k)}_1 &=& 1,\> 1 \leq k \leq n;
\nonumber \\
P^{(k)}_j &=& {k\choose j}(-1)^{j-1}\frac{1 + (k-1)\theta/2}
{(k-1)(1+k\theta/2)},\> 
2 \leq j \leq n.
\label{QP:0}
\end{eqnarray}
%
%
%
\end{thm}
\begin{proof}
The decomposition comes from identifying the coefficients of powers of $p(t)$ in (\ref{An:j},\ref{An:1},\ref{An:0}), since $e^{-\lambda_k t}= e^{-t}p(t)^k$, $k \geq 2$. $e^{-\lambda_1t}$ also appears in (\ref{An:1}) and (\ref{An:0}).
The expansion (\ref{spectral:0}) is upper triangular if $2 \leq j \leq n$, but contains a non-zero term when $k=0$ if $j=1$.
If $j\geq 2$, from (\ref{An:1}),
\[
Q_n^{(k)}P_j^{(k)} = 
{n\choose k}{k\choose j}(-1)^{k-j}
\]
If $j=1$, from (\ref{An:1}),
\begin{eqnarray*}
Q_n^{(k)}P_1^{(k)} &=& n(-1)^{k-1}{n-1\choose k-1}
 - {n\choose k}(-1)^{k-1}\frac{1}{1+(k-1)\theta/2},\>k>1
 \nonumber \\
 &=& {n\choose k}(-1)^{k-1}\frac{(k-1)(1+k\theta/2)}{1+(k-1)\theta/2}
 \nonumber \\
 Q_n^{(1)}P_1^{(1)} &=& \sum_{i=1}^{n}{n\choose i}(-1)^{i-1}
 \frac{1}{1+(i-1)\theta/2},
 \end{eqnarray*}
 and if $j=0$
\begin{eqnarray*}
  Q_n^{(k)}P_0^{(k)} &=& {n\choose k}(-1)^k\frac{(k-1)\theta/2}{1+(k-1)\theta/2},\>2 \leq k \leq n
 \nonumber \\
 Q_n^{(1)}P_0^{(1)} &=& -\sum_{i=1}^n{n\choose i}(-1)^{i-1}\frac{1}{1 + (i-1)\theta/2}.
\end{eqnarray*}
The definitions of $Q_n^{(k)}$ and $P_j^{(k)}$ in (\ref{QP:0}) correctly satisfy these products.
\end{proof}
\subsection{Population replacement}
Forward in time the times at which the population is replaced by the offspring of a single individual form a renewal process $\{T_j\}_{j=1}^\infty$ with independent exponential increments $S_j=T_j-T_{j-1}$, $j \geq 0$ ($T_0=0$). We find the density of type 1 individuals when $k$ population replacements have taken place and express the transition density $f(\xi;x,t)$ as a sum of these densities.
\begin{thm} Let $f_k(\xi;x,t)$ be the joint probability that there are $k\geq 1$ population replacements in $(0,t)$ and the density of the frequency of type 1 individuals at time $t$, when initially the frequency is $x$. Then
\begin{equation}
f_k(\xi;x,t) = 
\begin{cases}
\frac{2kr_1(\xi;t)}{\theta t(\xi-p)}
\Bigg (1+\frac{2}{\theta t}
\log\Big (\frac{\xi - p}{1-p}\Big )\Bigg )^{k-1}
(t^k/k!)e^{-t},
&\xi > p + (1-p)e^{-\theta t/2}\\
\frac{2kr_2(\xi;t)}{\theta t(p-\xi)}
\Bigg (1+\frac{2}{\theta t}\log\Big (\frac{p-\xi}{p}\Big )\Bigg )^{k-1}
(t^k/k!)e^{-t},
&\xi < p(1-e^{-\theta t/2}),
 \end{cases}
\label{density:k}
\end{equation}
where
\[
r_1(\xi;t) = p+(x-p)e^{-\theta t/2}\frac{1-p}{\xi-p},\>
r_2(\xi;t) =1-p-(x-p)e^{-\theta t/2}\frac{p}{p-\xi}.
\]
\end{thm}
\begin{proof}
Let $P(s) =\big (p_{ij}(s)\big )_{i,j=1}^2$ be a transition matrix for the probability of a replacement of a population of type $i$ individuals by type $j$ individuals at time $s$. The entries of $P(s)$ are those in Theorem \ref{theorem:1}. That is
\begin{equation*}
P(s) =
\begin{bmatrix}
p + (1-p)e^{-\theta s/2}&(1-p)(1-e^{-\theta s/2})\\
p(1-e^{-\theta s/2})&1-p + pe^{-\theta s/2}
\end{bmatrix}.
\end{equation*}
The left and right eigenvector matrices of $P(s)$ are
\[
U = 
\begin{bmatrix}
p&1-p\\
1&-1
\end{bmatrix},\>\>
U^{-1} = 
\begin{bmatrix}
1&1-p\\
1&-p
\end{bmatrix}
\]
with eigenvalues $1,e^{-\theta s/2}$. Therefore 
\[
P(s)= U^{-1}\text{Diag}\big (1,e^{-\theta s/2}\big )U,\>\>
\prod_{j=2}^kP(s_j)=U^{-1}\text{Diag}\big (1,e^{-\theta (t_k-t_1)/2}\big )U,
\]
where the product is taken from $P(s_1)$ on the left to $P(s_k)$ on the right.
At $t_1^{+0}$ the population replacement is of type 1 with probability $q_1(t_1;x)$.
Thus the probability of type 1 or 2 replacements at time $t_k$
conditional on $\{t_j\}_{j=1}^k$ is
\begin{eqnarray}
R\big (x;\{t_j\}_{j=1}^k\big )&:=&
\begin{bmatrix}
q_1(t_1;x),&1-q_1(t_1;x)
\end{bmatrix}
\prod_{j=2}^kP(s_j)\nonumber \\
&=& 
\begin{bmatrix}
p+(x-p)e^{-t_k\theta/2},
1-p-(x-p)e^{-t_k\theta/2}
\end{bmatrix}.
\end{eqnarray}
Then the conditional frequency of type 1 individuals at time $t$ is 
\begin{eqnarray}
&&p_{11}(t-t_k)\text{~with~probability~}r_1\big (x;\{t_j\}_{j=1}^k\big )
\nonumber \\
&&p_{21}(t-t_k)\text{~with~probability~}r_2\big (x;\{t_j\}_{j=1}^k\big ),
\label{cdl:0}
\end{eqnarray}
where
 $R\big (x;\{t_j\}_{j=1}^k\big )
= \big [r_1\big (x;\{t_j\}_{j=1}^k\big ),r_2\big (x\{t_j\}_{j=1}^k\big )\big ].
$
The probabilities in (\ref{cdl:0}) only depend on the time $t_k$.
$T_k$ has a probability distribution, conditional on $k$ replacements in $(0,t)$, of 
\begin{equation}
k(t_k/t)^{k-1}t^{-1},\>\>0 < t_k < t.
\label{cdl:1}
\end{equation}
A similar calculation to that in Theorem \ref{theorem:1} gives that the probability of having $k\geq 1$ replacements in $(0,t)$ and density of $\xi(t)$ is (\ref{density:k}).
There is an expansion of the transition density
\begin{equation}
f(\xi;x,t) = \sum_{k=1}^\infty f_k(\xi;x,t).
\label{trexp:0}
\end{equation}
It must be that (\ref{trexp:0}) holds and a check shows this to be true.
\end{proof}
\section{The star-shaped coalescent with $d$ types}\label{Section 4}
In this section an extension is made to $d$ types, then the infinitely-many-sites limit as $d \to \infty$ is taken.
The model is a star-shaped coalescent with mutation occurring on the lineages at rate $\theta/2$, with the type chosen as $i$ with probability $p_i=\theta_i/\theta$, $i=1,\ldots,d$. 
 Analysis of the model in Theorem \ref{theorem:7} follows that when $d=2$ in Section \ref{Section 2}. The proof is omitted because of similarity with Theorem \ref{theorem:1}.
\begin{thm}\label{theorem:7}
Let $\{\bm{\xi}(t)\}_{t\geq 0}$ be the frequency of $d$ types of individuals at time $t$, when $\bm{\xi}(0)=\bm{x}$, in the Fleming-Viot process with a generator acting on functions $g(\bm{x}) \in C^2(\Delta_d)$,
$\Delta_d = \{\bm{x};x_1+\cdots+x_d=1\}$, of
\begin{equation}
{\cal L}g(\bm{x}) =
\sum_{i=1}^dx_i\big (g(\bm{e}_i) - g(\bm{x})\big )
+ (\theta/2)\sum_{i=1}^d(p_i-x_i)g_i(\bm{x}).
\label{gend:0}
\end{equation}
Denote $p_{ij}(t)$, $i,j = 1,\ldots d$ as the probability of a change of type in a single line of the star-shaped coalescent forward in time from $i$ to $j$ in time $t$ and $q_i(t;\bm{x})$ as the probability a single gene chosen at time $t$ forward in time is of type $i = 1,\ldots,d$. Specifically
\[
p_{ij}(t) = \delta_{ij}e^{-\theta t/2} + (1-e^{-\theta t/2})p_j,
\]
\[
q_i(t;\bm{x}) = x_ie^{-\theta t/2}+(1-e^{-\theta t/2})p_i.
\]
Then 
\begin{equation}
\bm{\xi}(t) = 
\begin{cases}
\bm{q}(t;\bm{x})&\text{~with~probability~}e^{-t}\\
\bm{p}_i(\tau)&\text{~with~probability~}(1-e^{-t})q_i(t-\tau;\bm{x}),\>i=1,\ldots,d,
\end{cases}
\end{equation}
where $\bm{p}_i(\tau)$ is the $i$th row of $P(\tau) = \big (p_{ij}(\tau)\big )_{i,j=1}^d$ and $\tau$ has a truncated exponential density (\ref{trunc}). The continuous part of the distribution of $\bm{\xi}(t)$ has a density in the $i$th region for $\xi_i(t)$ of
\begin{equation}
\Big (p_i + (x_i-p_i)e^{-\theta t/2}\frac{1-p_i}{\xi_i-p_i}\Big )
\frac{2}{\theta}
\Big (\frac{\xi_i - p_i}{1 - p_i}\Big )^{2/\theta -1}
\frac{1}{1-p_i}
\>\>,\xi_i > p_i + (1-p_i)e^{-\theta t/2},
\end{equation}
and 
\[
\xi_j(t) = \Big (1 - \frac{\xi_i(t) - p_i}{1 - p_i}\Big )p_j,\>\>j\ne i.
\]
The stationary density in the $i$th region is
\begin{equation}
p_i\frac{2}{\theta}
\Big (\frac{\xi_i - p_i}{1 - p_i}\Big )^{2/\theta -1}
\frac{1}{1-p_i},\>\>\xi_i > p_i
\label{dstationary:0}
\end{equation}
and
\[
\xi_j= \Big (1 - \frac{\xi_i - p_i}{1 - p_i}\Big )p_j,\>\>j\ne i.
\]
\end{thm}
\begin{rem}
There is a similar representation to that in Remark \ref{Remark 1} for the stationary distribution (\ref{dstationary:0}) of $\bm{\xi}$. 
Let $\eta$ be a random variable with density $(2/\theta)\eta^{2/\theta -1}$, $0 < \eta < 1$. Then for $i=1,\ldots,d$, with probability $p_i$,
\begin{eqnarray}
\xi_i &=& (1-\eta)p_i + \eta,
\nonumber \\
\xi_j &=& (1-\eta)p_j,\>\>j\ne i.
\label{drep:0}
\end{eqnarray}
\end{rem}
\begin{rem}
The random variables $\{\xi_i\}$ are exchangeable if $p_i=1/d$, $i=1,\ldots ,d$ with one taking the value $(1-\eta)/d + \eta$ and the $d-1$ other random variables taking the value $(1-\eta)/d$. In the limit as $d\to \infty$ there is one frequency of $\eta$ and an infinite number of \emph{dust} frequencies of total mass $1 - \eta$. The limit sampling distribution in $n$ unlabelled individuals is
\begin{equation}
{n \choose j}\mathbb{E}\big [\eta^j(1-\eta)^{n-j}\big ]
= \frac{n!}{j!}\cdot \frac{(2/\theta)_{(j)}}{(1+2/\theta)_{(n)}},\label{sampling:0}
\end{equation}
where there is one sampling frequency of $j$ and $n-j$ distinct frequencies of 1. Straightforwardly there is one type with probability 
$\mathbb{E}\big [\eta^n\big ]$ and $k$ types with probability 
${n\choose k-1}\mathbb{E}\big [\eta^{n-k+1}(1-\eta)^{k-1}\big ]$. This sampling distribution agrees with the sampling distribution (27) in \citet{M2006} after simplification and identification of parameters.
\end{rem}
\subsubsection{A general mutation model}
A similar theorem to Theorem \ref{theorem:7} holds for the general Markov mutation model without the exact detail of the exact transition density and stationary density. The mutation model is that mutations occur along lines at rate $\theta/2$ and, forward in time changes of type are made from $i$ to $j$ according to a transition probabiity matrix $\bm{P}$. It is assumed that $\bm{P}$ is recurrent with stationary distribution $\bm{\gamma}$.
\begin{thm}\label{theorem:8}
Let $\{\bm{\xi}(t)\}_{t\geq 0}$ be the frequency of $d$ types of individuals at time $t$, when $\bm{\xi}(0)=\bm{x}$, in the Fleming-Viot process with a generator acting on functions $g(\bm{x}) \in C^2(\Delta_d)$,
$\Delta_d = \{\bm{x};x_1+\cdots+x_d=1\}$, of
\begin{equation}
{\cal L}g(\bm{x}) =
\sum_{i=1}^dx_i\big (g(\bm{e}_i) - g(\bm{x})\big )
+ (\theta/2)\sum_{i=1}^d\Big (\sum_{j=1}^d(x_jp_{ji}-x_i)\Big )g_i(\bm{x}).
\label{ggend:0}
\end{equation}
Denote $p_{ij}(t)$, $i,j = 1,\ldots d$ as the probability of a change of type in a single line of the star-shaped coalescent forward in time from $i$ to $j$ in time $t$ and $q_i(t;\bm{x})$ as the probability a single gene chosen at time $t$ forward in time is of type $i = 1,\ldots,d$. Specifically
\[
\bm{P}(t) = \exp \left \{(\theta/2)(\bm{P}-\bm{I})t\right \},\>\>
q_i(t;\bm{x}) = \sum_{k=1}^dx_kp_{ki}(t).
\]
Then 
\begin{equation}
\bm{\xi}(t) = 
\begin{cases}
\bm{q}(t;\bm{x})&\text{~with~probability~}e^{-t}\\
\bm{p}_i(\tau)&\text{~with~probability~}(1-e^{-t})q_i(t-\tau;\bm{x}),\>i=1,\ldots,d,
\end{cases}
\end{equation}
where $\bm{p}_i(\tau)$ is the $i$th row of $\bm{P}(\tau)$ and $\tau$ has a truncated exponential density (\ref{trunc}). If $\bm{\xi}$ is a random variable with the stationary distribution of $\{\bm{\xi}(t)\}_{t\geq 0}$, then 
\begin{equation}
\bm{\xi} =^{\cal D}\bm{p}_i(\tau)\text{~with~probability~}\gamma_i,\>i=1,\ldots,d
\label{generalstat:0}
\end{equation}
where $\tau$ has an exponential distribution of rate 1.
\end{thm}

\section{The star-shaped coalescent with general frequency dependent change}
Consider a population of individuals of types 1 or 2 evolving in time and let $\xi(t)$ be the frequency of type 1 individuals at time $t$, when $\xi(0)=x$. $\{\xi(t)\}_{t\geq 0}$ is a Fleming-Viot process with with generator
\begin{equation}
{\cal L}g(x) = x\big [g(1)-g(x)\big ] + (1-x)\big [g(0) - g(x)\big ] + v(x)g^\prime (x).
\label{gen:0a}
\end{equation}
At time $t$ after a replacement and before the next replacement the frequency of type 1 individuals is $\chi (t)$ satisfying
\begin{equation}
\frac{d}{dt}\chi (t) = v(\chi(t))
\label{determ:0}
\end{equation}
with $\chi (0) = 1$ or $0$ depending on the type of the replacement being respectively 1 or 2.
$\{\xi(t)\}_{t\geq 0}$ is a jump process such that when there is a jump the population is completely replaced by either all type 1 genes or all type 2 genes. The time $T$ between jumps has an exponential distribution with parameter 1. Between these replacements $\xi(t)$ changes deterministically governed by $v(x)$. Typically the model is one with mutation and selection where
\begin{equation}
v(x) = \cfrac{1}{2}(\theta_1 - \theta x) + \cfrac{1}{2}\beta x(1-x).
\label{mutsel:0}
\end{equation}
with $\chi (0) = 1$ or $0$ depending on the type of the replacement being respectively 1 or 2.

For simplicity we restrict attention to two types and the stationary distribution of $\{\xi(t)\}_{t\geq 0}$ in this section, though there are also analogous results to the case with no selection when there are $d$ types and there is also a solution for the transition functions of $\{\xi(t)\}_{t\geq 0}$. 
\begin{thm}
Let $\mu(t)$ be the solution to (\ref{determ:0}) when a type 1 replacement occurs with $\mu(0)=1$ and $\nu(t)$ the solution when a type 2 replacement occurs with $\nu(0)=0$. Let $P=(p_{ij})$ be the transition matrix between jumps of type $i,j \in \{1,2\}$.  Then 
\begin{equation}
P= \begin{pmatrix}
 \mathbb{E}[\mu(T)]&\mathbb{E}[1-\mu(T)]\\
 \mathbb{E}[\nu(T)]&\mathbb{E}[1-\nu(T)]
 \end{pmatrix}.
\label{trmatrix:0}
\end{equation}
$T$ is exponential with rate 1, and $P$ describes then skeleton transition matrix between the states of the population at instants after replacement when the population is composed of just one type.
The stationary distribution of the type of replacement $(\pi_1,\pi_2)$ is the stationary distribution $(\pi_1,\pi_2)$ of $P$,
that is
\begin{equation}
\pi_1 = \frac{\mathbb{E}[\nu(T)]}{\mathbb{E}[\nu(T)+ 1-\mu(T)]}.
\label{stat:g}
\end{equation}
The stationary distribution of $\{\xi(t)\}_{t\geq 0}$ is the distribution of $\mu(\tau)$ if the last replacement was type 1 or $\nu(\tau)$ it was type 2, where $\tau$ is an exponential random variable of rate 1. 
A random variable $\xi$ with this stationary distribution has a density
\begin{equation}
f(\xi;v) = \pi_1 e^{-\mu^{-1}(\xi)}
\frac{1}{|\mu^\prime ( \mu^{-1}(\xi) )|}
+ \pi_2e^{-\nu^{-1}(\xi)}
\frac{1}{|\nu^\prime(\nu^{-1}(\xi))|}.
\label{general_d:0}
\end{equation}
The mean $\mathbb{E}\big [\xi\big ] = \pi_1$.
\end{thm}
\begin{proof}
The frequency of type 1 individuals at the instant before replacement is either $\mu(T)$ or $\nu(T)$ depending on whether the prior replacement was type 1 or 2. A choice of a type 1 or 2 replacement is then either 
 $\mathbb{E}\big [\mu(T)\big ]$ or $\mathbb{E}\big [\nu(T)\big ]$ which is the first column of $P$. The second column follows similarly. The stationary distribution is clearly $(\pi_1,\pi_2)$. The times of replacement form a renewal process with exponentially distributed increments $T$ of rate 1. The asymptotic distribution of the time from the last replacement to a given time point between two replacements which cover the point is again exponential $\tau$ of rate 1. Therefore a random variable $\xi$ with the stationary distribution satisfies
 \begin{equation}
 \xi =^{\cal D}
 \begin{cases}
 \mu(\tau)&\text{with~probability~}\pi_1,\\
 \nu(\tau)&\text{with~probability~}\pi_2.
 \end{cases}
 \label{repxx}
 \end{equation}
 The density is then (\ref{general_d:0}) after a change of variable $\tau \to \xi$. The inverse functions $\mu^{-1}(\xi)$ and $\nu^{-1}(\xi)$ may have disjoint support regions, such as in (\ref{density:2}).
The mean of $\xi$ being $\pi_1$ follows from (\ref{repxx}).
\end{proof}
\subsection{Fixation probabilities in a model with selection}
If there is no stationary distribution then the probability of fixation of type 1 individuals is the probability that the first jump is of type 1. Let $P_1(x)$ be the probability of fixation beginning at $x$.
Consider the particular selective model where there is no mutation and  $v(x)=(\beta/2)x(1-x)$.
Then
\begin{equation}
P_1(x) = \mathbb{E}\big [\chi(T)\big ],
\label{fix:0}
\end{equation}
where $\chi$ satisfies (\ref{determ:0}) with $\chi(0) = x$.
The logistic curve solution to (\ref{determ:0}) is 
\[
\chi(t) = \frac{x}{(1 - x)e^{-\beta t/2} + x}
\]
and 
\begin{equation}
P_1(x) = \int_0^\infty
e^{-t}\chi (t)dt
=
(2/\beta)x\int_0^1\frac{z^{2/\beta - 1}}{1 - (1-x)(1-z)}dz.
\label{P1:0}
\end{equation}
The probability fixation of type 2 individuals, from an initial frequency of $y=1-x$ is, from (\ref{P1:0}), 
\begin{equation}
P_2(y) = 1 - (2/\beta)(1-y)\int_0^1\frac{z^{2/\beta -1}}{1 - y(1-z)}dz.
\label{P2:0}
\end{equation}

\subsection{Mutation and selection}
If there is mutation between types 1 and 2 then a stationary distribution exists.
\begin{thm}
In a model with mutation and selection 
\begin{equation*}
v(x) = \cfrac{1}{2}(\theta_1 - \theta x) + \cfrac{1}{2}\beta x(1-x),
\end{equation*}
where $\theta >0$, $\beta > 0$.
Let $0 < p = \theta_1/\theta < 1$, $\phi = \theta/\beta$ and $0 < r_1 < 1$, $r_2 < 0$ be the two (real) roots of
\begin{equation}
\chi(1-\chi) + \phi(p-\chi)= 0,
\label{roots}
\end{equation}
specifically
\[
r_1,r_2=\frac{1}{2}\Big ( 1-\phi \pm \sqrt{(1-\phi)^2+4\phi p} \Big ).
\]
Then the probability of replacement transitions at time $t$, $P(t)$, has elements
\begin{eqnarray}
p_{11}(t) = \mu(t) &=& r_1 + (r_1-r_2)\frac{\frac{1-r_1}{1-r_2}e^{-\beta (r_1-r_2)t/2}}
{1 - \frac{1-r_1}{1-r_2}e^{-\beta (r_1-r_2)t/2}}\nonumber \\
p_{21}(t) = \nu(t) &=& r_1 + (r_1-r_2) \frac{\frac{r_1}{r_2}e^{-\beta (r_1-r_2)t/2}
}
{1 - \frac{r_1}{r_2}e^{-\beta (r_1-r_2)t/2}},
\label{selmut:1}
\end{eqnarray}
and $P=\mathbb{E}\big [P(T)\big ]$.
The stationary density (\ref{general_d:0}) is
\begin{eqnarray}
&&\phantom{+}\pi_1\frac{2}{\beta}
\Bigg (\frac{\xi-r_1}{\xi-r_2}\cdot \frac{1-r_2}{1-r_1}\Bigg )
^{\frac{2}{\beta(r_1-r_2)} - 1}
\frac{1}{(\xi-r_2)^2}\cdot\frac{1-r_2}{1-r_1}I\{r_1 < \xi < 1\}~~~
\nonumber \\
&&+\pi_2\frac{2}{\beta}
\Bigg (\frac{r_1-\xi}{\xi-r_2}\cdot\frac{-r_2}{r_1}\Bigg )
^{\frac{2}{\beta} - 1}
\frac{1}{(r_1-\xi)^2}\cdot\frac{-r_2}{r_1}I\{0 < \xi < r_1\}.
\label{selection:den}
\end{eqnarray}
\end{thm}
\begin{proof}
The differential equation (\ref{determ:0}) for $\chi(t)$ is 
\begin{equation}
\frac{d}{dt}\chi = \cfrac{1}{2}\Big (\theta (p- \chi)
 + \beta \chi(1-\chi)\Big ).
\label{mutsel:1a}
\end{equation}
To obtain the solution of (\ref{mutsel:1a}) let $\phi=\theta/\beta$ and factorize the quadratic term
\begin{equation}
Q(\chi) := -\chi^2 + (1-\phi)\chi + p\phi.
\label{quadratic}
\end{equation}
Let $r_1,r_2$ be the roots of (\ref{roots}) in the statement of the theorem.
$Q(0)=p\phi > 0$, $Q(1) = -(1-p)\phi < 0$ and $Q(\chi) \to -\infty$ as 
$\chi \to -\infty$ so $0 < r_1 < 1$ and $r_2 < 0$.
(\ref{mutsel:1a}) can be written as
\[
\frac{1}{r_1-r_2}\Big (\frac{1}{\chi - r_1} - \frac{1}{\chi - r_2}\Big )
\frac{d}{dt}\chi = -\frac{\beta}{2}.
\]
Therefore
\begin{equation}
\frac{r_1-\chi(t)}{\chi(t) - r_2}=\frac{r_1-\chi(0)}{\chi(0) - r_2}
e^{-\beta (r_1-r_2)t/2}.
\end{equation}
Let 
\[
C(\chi(0);t) = \frac{r_1-\chi(0)}{\chi(0) - r_2}e^{-\beta (r_1-r_2)t/2}.
\]
The solution of (\ref{mutsel:1a}) can be written as
\begin{equation}
\chi(t)
= r_1 - (r_1-r_2)\frac{C(\chi(0);t)}{1+C(\chi(0);t)}.
\end{equation}
Now $\mu(t)=\chi(t)$ when $\chi(0)=1$ and $\nu(t)=\chi(t)$ when $\chi(0)=0$; which gives (\ref{selmut:1}). Note that $C(1;t) \leq 0$ and $C(0;t) > 0$.
To find the stationary density make a change of variable $\tau \to \xi$ where $\tau$ is exponential with rate 1 and $\xi = \chi(\tau)$.
The density, with $\chi(0)$ either 1 or 0, is
\begin{equation}
\frac{2}{\beta}\Bigg (
\frac{r_1-\xi}{\xi - r_2}\cdot \frac{\chi(0)-r_2}{r_1-\chi(0)}
\Bigg )^{\frac{2}{\beta (r_1-r_2)}}
\frac{1}{(\xi-r_2)|\xi-r_1|},
\label{selection:den:1}
\end{equation}
where the support of the distribution depends of the positivity of the expression in large brackets.
Evaluating (\ref{selection:den:1}) in the two cases gives (\ref{selection:den}).

\end{proof}
\begin{cor}
Let $a=(r_1-r_2)/2 >0$, $0 < b = (1-r_1)/(1-r_2) < 1$, and $c=-r_1/r_2 > 0$.
Then
\begin{eqnarray} 
p_{11} &=&
r_1 + 2b^{-1/(a\beta)}\beta^{-1}\int_0^by^{1/(a\beta)}(1-y)^{-1}dy
\label{pseries:0}
\\
&=&
r_1 + 2\beta^{-1}\sum_{k=0}^\infty \frac{b^{k+1}}{1/(a\beta) + k + 1};
\nonumber \\
p_{21} &=& r_1 - 2c^{-1/(a\beta)}\beta^{-1}\int_0^{c/(1+c)}y^{1/(a\beta)}(1-y)^{-1/(a\beta)-1}dy
\label{pseries:1}
\\
&=&
r_1 - 2c^{-1/(a\beta)}\beta^{-1}\sum_{k=0}^\infty
\frac{\big ((a\beta+1)/(a\beta)\big )_{(k)}}{k!}
\frac{\Big (c/(1+c)\Big )^{1/(a\beta)+k+1}}{1/(a\beta)+k+1} 
\nonumber
\end{eqnarray}
$\pi_1$ and $\pi_2$ can be computed from $p_{11}$ and $p_{21}$.
\end{cor}
\begin{proof}
The inequalities for $a,b,c$ are elementary. Let $\tau$ be an exponential
random variable with rate 1.
\begin{eqnarray}
p_{11} &=& \mathbb{E}\big [p_{11}(\tau)\big ]
\nonumber \\
&=& r_1 + 2a\int_0^\infty e^{-t}
\frac{be^{-a\beta t}}{1 + be^{-a\beta t}}dt;
\label{pseries:0a}
\\
p_{21} &=& \mathbb{E}\big [p_{21}(\tau)\big ]
\nonumber \\
&=& r_1 -2a\int_0^\infty e^{-t}\frac{ce^{-a\beta t}}{1 + ce^{-a\beta t}}dt.
\label{pseries:1a}
\end{eqnarray}
Now make the changes of variable $y=be^{-a\beta t}$ in (\ref{pseries:0a}) to obtain (\ref{pseries:0}) and $y=ce^{-a \beta t}/(1+ce^{-a \beta t})$ in (\ref{pseries:1a}) to obtain (\ref{pseries:1}). 
\end{proof}
In a model with no selection with $\beta = 0$, $p_{11}=p_{21}=p$. Letting $\beta \to 0$ gives $r_1 \to p$, $r_2 \to -\infty$, $r_2\beta \to -\theta$
with (\ref{selmut:1}), (\ref{pseries:0}) and (\ref{pseries:1}) converging to the correct neutral quantities.
\subsection{Ancestral selection graph}
\subsubsection{Duality}
There is an ancestral selection graph which is dual to a process $\{\xi(t)\}_{t\geq 0}$ with a generator described by
\begin{equation}
{\cal L}g(x) = x\big (g(1)-g(x)\big ) + (1-x)\big (g(0) - g(x)\big )
- \frac{\beta}{2}x(1-x)\frac{\partial}{\partial x},
\label{asg:0}
\end{equation}
where $\beta >0$ and $g \in C^2([0,1])$. (If $\beta < 0$ then consider the process $\{1-\xi(t)\}_{t\geq 0}$.) For $i \in \mathbb{Z}_+$ 
\begin{equation}
{\cal L}x^i = x - x^i + \frac{i\beta}{2}(x^{i+1}-x^i).
\label{asg:1}
\end{equation}
Re-interpreting (\ref{asg:1}) as a generator acting on $i$, the dual process $\{B_n(t)\}_{t\geq 0}$ with $B_n(0)=n$ is a Markov process on $\mathbb{Z}_+$ with rates
\begin{equation}
r_{ij} = 
\begin{cases}
\frac{i\beta}{2},&j=i+1,\\
1,&j=1.
\end{cases}
\label{asgrates:0}
\end{equation}
The duality equation is, for $n\in \mathbb{Z}_+$,
\begin{equation}
\mathbb{E}_x\big [\xi(t)^n\big ] = \mathbb{E}_n\big [x^{B_n(t)}\big ],
\label{sduality:0}
\end{equation}
where expectation is on the left with respect to $\xi(t)$ and on the right with respect to $B_n(t)$.
$\{B_n(t)\}_{n=1}^\infty$ does not describe the full ancestral selection graph, but it is a lineage counting process.
 Duality is studied in the Wright-Fisher model with selection in \cite{M2009} and in a $\Lambda$-Fleming-Viot model in \citet{EGT2010}. Duality is much simpler here and we have explicit results.
\subsubsection{Stationary distribution}
There is a stationary distribution $\bm{\pi}$ for the edges in the ancestral selection graph calculated from the rates (\ref{asgrates:0}) which satisfies
\begin{equation*}
\pi_i = 
\frac{(i-1)!}{(2/\beta + 1)_{(i)}}\pi_1,\>i \geq 1.
\end{equation*}
$\pi_1=2/\beta$ is evaluated from
\begin{eqnarray*}
1 &=& \sum_{i=1}^\infty \pi_i\\
&=& \pi_1\sum_{i=1}^\infty \frac{\Gamma (i)\Gamma (2/\beta+1)}{\Gamma(2/\beta+1+i)}\\
&=& \pi_1\sum_{i=1}^\infty \int_0^1z^{i-1}(1-z)^{2/\beta}dz\\
&=& \pi_1\int_0^1(1-z)^{2/\beta-1}dz\\
&=&(\beta/2)\pi_1.
\end{eqnarray*}
Therefore
\begin{equation}
\pi_i = 
\frac{(2/\beta)(i-1)!}{(2/\beta + 1)_{(i)}},\>i \geq 1.
\label{selstat:0}
\end{equation}
Letting $t \to \infty$ in the duality equation (\ref{sduality:0}) 
\begin{equation*}
\mathbb{E}_y\big [\xi(\infty)^n\big ] = \mathbb{E}_n\big [y^{B_n(\infty)}\big ] = P_2(y),
\end{equation*}
taking into account the negative selection coefficient in (\ref{asg:0}).
Therefore a generating function for $B_n(\infty)$, which has a stationary distribution $\{\pi_i\}_{i=1}^\infty$ is
\begin{equation*}
\sum_{i=1}^\infty\pi_iy^i = P_2(y).
\end{equation*}
Calculating the coefficient of $y^i$  in $P_2(y)$ from (\ref{P2:0}) shows this to be true. This argument is used in \citet{M2009} for the Wright-Fisher diffusion with selection.
\subsubsection{Time to the ultimate ancestor} 
The ultimate ancestor is reached when $B_n(T_{\text{UA}})=1$ for the first time at a random time $T_{\text{UA}}$. The ultimate ancestor is reached with probability one, because the probability that $B_n(t) > 1$ for all $t \geq 0$ is 
\[
\lim_{N\to \infty} \prod_{j=n}^N\Big (1- \frac{1}{1+j\beta/2}\Big ) = 0.
\]
In fact the ultimate ancestor is always the most recent common ancestor in this star-shaped coalescent unlike in the Wright-Fisher model with selection where the event depends on the type of the ultimate ancestor.
Let $G_n(\varphi)$ be the Laplace transform of the time to the ultimate ancestor $T_{\text{UA}}$ from $B_n(0)=n$ individuals. It must be that 
$G_n(\varphi) = (1+\varphi)^{-1}$, $n\geq 2$. We give a formal proof of this.
$\{G_n(\varphi)\}_{n\geq 2}$ satisfies the system of equations
\begin{eqnarray}
G_n(\varphi) &=& \frac{1+n\beta/2}{\varphi +1 + n\beta/2}
\Big (\frac{1}{1+n\beta/2} + \frac{n\beta/2}{1+n\beta/2}G_{n+1}(\varphi)\Big )
\nonumber \\
&=&
\frac{1}{\varphi +1+ n\beta/2} +  \frac{n\beta/2}{\varphi +1+n\beta/2}G_{n+1}(\varphi),
\>n\geq 2
\label{weq:0}
\end{eqnarray}
with $G_1(\varphi)=1$, from a decomposition according to the first jump.
Rewrite and extend (\ref{weq:0}) for $N > n$ as
\begin{eqnarray}
(1+\varphi)G_n(\varphi) - 1 &=& \frac{n\beta/2}{\varphi + 1+n\beta/2}\big ((1+\varphi)G_{n+1}(\varphi) - 1\big )
\nonumber \\
&=&
\prod_{j=n}^N\Big ( 1 - \frac{1+\varphi}{\varphi +1 +j\beta/2}\Big )
\big ((1+\varphi)G_{N+1}-1\big ).
\label{weq:1}
\end{eqnarray}
Now let $N \to \infty$ in (\ref{weq:1}).  $(1+\varphi)G_{N+1}(\varphi) - 1$ is bounded and
\[
 \prod_{j=n}^N\Big ( 1 - \frac{1+\varphi}{\varphi +1+j\beta/2}\Big ) \to 0,
\]
so  $G_n(\varphi) = (1+\varphi)^{-1}$.
\section{Acknowledgement} Robert Griffiths visited \emph{The Institute of Statistical Mathematics} in Tachikawa, Tokyo, for three months in 2014 where this research was developed. He thanks the Institute for its support and hospitality.

\begin{thebibliography}{99}
\bibitem[Berestycki(2009)]{B2009} {\sc Berestycki, N.} 2009.
Recent progress in coalescent theory.  \emph{Ensaios Matematicos} {\bf 16} 1--193.
%
\bibitem[Bertoin and Le Gall(2006)]{BLG2006}
{\sc Bertoin J. and  Le Gall J.-F.} 2006. Stochastic flows associated to coalescent processes
III: Limit theorems. \emph{Illinois J. Math} {\bf 50} 147--181.
%
%
\bibitem[Birkner and Blath(2009)]{BB2009} {\sc Birkner, B. and Blath, J.}  2009. Measure-Valued diffusions, general coalescents and population genetic inference. Chapter 12 in: \emph{Trends in Stochastic Analysis}, Cambridge University Press.

%
\bibitem[Donnelly and Kurtz(1999)]{DK1999} {\sc Donnelly, P.J. and Kurtz, T.G.} 1999. Particle representations for measure-valued population models. \emph{Ann. Probab.} {\bf 24} 166--205.
%
\bibitem[Durrett(2008)]{D2008}
{\sc Durrett, R.} 2008. \emph{Probability models for DNA sequence evolution.} 2nd ed. Springer Science \& Business Media, New York.
%
\bibitem[Eldon and Wakeley(2006)]{EW2006} {\sc Eldon, B. and Wakeley, J.} 2006. Coalescent processes when the distribution of offspring number among individuals is highly skewed. \emph{Genetics} {\bf 172} 2621--2633.
%
\bibitem[Etheridge et~al.(2010)]{EGT2010}
{\sc Etheridge, A. M., Griffiths, R. C., Taylor, J. E.} 2010. A coalescent dual process in a Moran model with genic selection, and the lambda coalescent limit. \emph{Theor. Popul. Biol.} {\bf 78} 77--92.

\bibitem[Etheridge(2012)]{E2012} {\sc Etheridge, A.} 2012. Some mathematical models from population genetics. {\'E}cole d'{\'E}t{\'e} de Probabilitit{\'e}s de Saint-Flour XXXIX-2009. Springer.
%
\bibitem[Gnedin et. al.(2014)]{GIM2014} {\sc Gnedin, A., Iksanov, A., Marynych, A.} 2014. $\Lambda$-coalescents: a survey. \emph{J. Appl. Probab.} {\bf 51A} 23--40. 
%
\bibitem[Griffiths(1980)] {G1980}
{\sc Griffiths, R. C.} 1980. Lines of descent in the diffusion approximation of neutral Wright-Fisher models. \emph{Theor. Popul. Biol.} {\bf 17} 37--50.
%
\bibitem[Griffiths and Tavar{\'e}(1998)]{GT1998}
{\sc Griffiths, R.C. and Tavar{\'e}, S.} 1998. The age of a mutation in a general coalescent tree. \emph{Stochastic Models} {\bf 14} 273--295.

\bibitem[Griffiths(2014)]{G2014}
{\sc Griffiths R. C.} 2014. The Lambda-Fleming-Viot process and a connection with Wright-Fisher diffusion. \emph{Adv. Appl. Prob.} {\bf 46} 1009--1035.

\bibitem[Kaneko(1988)]{K1988}
{\sc Kaneko, A.} 1988. Introduction to Hyperfunctions. Kluwer Academic Publishers, New York.

\bibitem[Mano(2009)]{M2009}
{\sc Mano, S.} 2009. Duality, ancestral and diffusion processes in models with selection. \emph{Theor. Popul. Biol.} {\bf 75} 164--175.

\bibitem[M{\"o}hle(2006)]{M2006} {\sc M{\"o}hle, M.} 2006. On sampling distributions for coalescent processes with simultaneous multiple collisions. \emph{Bernoulli} {\bf 12} 35--53.

\bibitem[Pitman(1999)]{P1999} {\sc Pitman, J.} 1999. Coalescents with multiple collisions. \emph{Ann. Probab.} {\bf 27} 1870--1902.

\bibitem[Pitman(2002)]{P2002} {\sc Pitman, J.} 2002. Combinatorial Stochastic Processes.
{\'E}cole d'{\'E}t{\'e} de Probabilitit{\'e}s de Saint-Flour XXXII-2002. Springer.


\bibitem[Sagitov(1999)]{S1999} {\sc Sagitov, S.} 1999. The general coalescent with asynchronous mergers of ancestral lines. \emph{J. Appl. Probab.} {\bf 36} 1116--1125.

\bibitem[Schweinsberg(2000)]{S2000} {\sc Schweinsberg, J.} 2000. A necessary and sufficient condition for the $\Lambda$-coalescent to come down from infinity. \emph{Electron. Comm. Probab.} {\bf 5} 1--11.


\bibitem[Tavar{\'e}(1984)]{T1984} 
{\sc Tavar{\'e}, S.} 1984. Line-of-descent and genealogical processes, and their application in population genetics models.  \emph{Theor. Popul. Biol.} {\bf 26} 119--164. 

\bibitem[Tellier and Lemaire(2014)]{TL2014}
{\sc Tellier, A. and Lemaire, C.} 2014. Coalescence 2.0: a multiple branching of recent theoretical developments and their applications. \emph{Mol. Ecology} {\bf 23} 2637--2652.




\end{thebibliography}
\end{document}